\documentclass[12pt]{amsart}
\usepackage{amsmath, amssymb,tabularx,graphicx}
\usepackage{cite}
\usepackage{url}
\usepackage[mathscr]{eucal}
\newtheorem{thm}{Theorem}[subsection]
\newtheorem{prop}[thm]{Proposition}
\newtheorem{lemma}[thm]{Lemma}
\newtheorem{defn}[thm]{Definition}

\newtheorem{rk}{Remark}
\newtheorem{question}[thm]{Question}
\newtheorem{ex}{Example}[subsection]
\newtheorem*{thma}{Theorem 3.1.5}
\newtheorem*{thmb}{Theorem 3.2.1}
\newtheorem*{thmc}{Theorem 4.2.1}
\newtheorem*{ack}{Acknowledgements}

\newcommand{\EE}{\mathbb{E}}
\newcommand{\ZZ}{\mathbb{Z}}

\newcommand{\RR}{\mathbb{R}}

\begin{document}

\title {Metrics on Visual Boundaries of CAT(0) Spaces} 
\author{ Molly A. Moran}
\thanks{The contents of this paper constitutes part of the author's dissertation for the degree of Doctor of Philosophy at the University of Wisconsin-Milwaukee under the direction of Professor Craig Guilbault.}

\begin{abstract} 
 A famous open problem asks whether the asymptotic dimension of a CAT(0) group
is necessarily finite. For hyperbolic $G$, it is known that
$\operatorname*{asdim}G$ is bounded above by $\dim\partial G+1$, which is
known to be finite. For CAT(0) $G$, the latter quantity is also known to be
finite, so one approach is to try proving a similar inequality. So far those
efforts have failed. 

Motivated by these questions we work toward understanding the relationship
between large scale dimension of CAT(0) groups and small scale dimension of
the group's boundary by shifting attention to the linearly controlled dimension of
the boundary. To do that, one must choose appropriate metrics for the
boundaries. In this paper, we suggest two candidates and develop some basic
properties. Under one choice, we show that linearly controlled dimension of
the boundary remains finite; under another choice, we prove that macroscopic
dimension of the group is bounded above by $2\cdot\ell$-$\dim \partial G+1$. Other useful
results are established, some basic examples are analyzed, and a variety of
open questions are posed.

\end{abstract}

\maketitle

\section{Introduction}

In \cite{Mo14} and \cite{GuMo15}, it was shown that coarse (large-scale) dimension properties of a space $X$ can impose restrictions on the classical (small-scale) dimension of boundaries attached to $X$. A natural question to ask is if the converse is true. For example, one might hope to use the finite-dimensionality of $\partial G$, proved first in \cite{Swe99} and following as a corollary of Theorem A in \cite{Mo14}, to attack the following well-known open question:
\begin{question} Does every CAT(0) group have finite asymptotic dimension?\end{question}

This question provides motivation for much of the work in what follows. Although we do not answer Question 1.0.1, a framework is developed that we expect will lead to future progress. Along the way, we prove some results that we hope are of independent interest; one such result is a partial solution to Question 1.0.1 that captures the spirit of our approach. 

As is often the case with questions about CAT(0) groups, Question 1.0.1 is rooted in known facts about hyperbolic groups. Gromov observed that all hyperbolic groups have finite asymptotic dimension. A more precise bound on the asymptotic dimension, which helps to establish our point of view, is the following:

\begin{thm}\cite{BuSc07,BuLe07} For a hyperbolic group, $\emph{asdim}G=\emph{dim}\partial G+1=\ell\emph{-dim}\partial G+1<\infty$. \end{thm}

In this theorem `$\operatorname*{asdim}$' denotes \emph{asymptotic dimension},
`$\dim$' denotes \emph{covering dimension}, and `$\ell$-$\dim$' denotes
\emph{linearly controlled dimension}. All of these terms will be explained in
Section 2.3. For now, we note that linearly controlled dimension is similar to, but
stronger than, covering dimension; both are small-scale invariants defined
using fine open covers. The difference is that $\ell$-$\dim$ is a metric
invariant, requiring a linear relationship between the mesh and the Lebesgue
numbers of the covers used. 


Implicit in the statement of Theorem 1.0.2 is that $\partial G$ be endowed with
a \emph{visual metric}. There is a family of naturally occurring visual
metrics on $\partial G$, but all are \emph{quasi-symmetric }to one-another.
That is enough to make $\ell$-$\dim\partial G$ well-defined. This also will be
explained shortly.

We can now summarize the content of this paper. We begin by reviewing a
number of key definitions and properties from CAT(0) geometry. Next, we recall definitions of quasi-isometry and
quasi-symmetry, and then we discuss variations, both small- and large-scale,
on the notion of dimension. To bring the utility of linearly controlled dimension to CAT(0) spaces, it is necessary to have specific metrics on their visual boundaries. Although CAT(0)
boundaries are important, well-understood, and metrizable, specific metrics
have seldom been used in a significant way. In Sections 3 and 4, we develop two
natural families of metrics for CAT(0) boundaries and verify a number of their
basic properties. One of these families $\left\{  d_{A,x_{0}}\right\}
_{x_{0}\in X}^{A>0}$ was discussed in \cite{Ka07}, where B. Kleiner asked whether the
induced action on $\partial X$ of a geometric action on a proper CAT(0) space
$X$ is ``nice''. After first showing that all metrics in the family $\left\{  d_{A,x_{0}%
}\right\}  _{x_{0}\in X}^{A>0}$ are quasi-symmetric in Section 3.1, we provide an affirmative answer to Kleiner's question with the following: 

\begin{thma}
Suppose $G$ acts geometrically on a proper CAT(0) space $X$, $x_{0}\in X$ and
$A>0$. Then the induced action of $G$ on $\left(  \partial X,d_{x_{0}%
,A}\right)  $ is by quasi-symmetries.
\end{thma}

In Section 3.2, we look to prove analogs of Theorem 1.0.2 for CAT(0) spaces. The
question of whether $\ell$-dimension of a CAT(0) group boundary agrees with
its covering dimension (under either of our metrics) is still open, but we can prove:

\begin{thmb}
If $G$ is a CAT(0) group, then $\left(  \partial G,d_{A,x_{0}}\right)  $ has
finite $\ell$-dimension.
\end{thmb}

As for the equality in Theorem 1.0.2, we are thus far unable to use the $\ell
$-dimension of $\left(  \partial X,d_{A,x_{0}}\right)  $ to make conclusions
about the asymptotic dimension of $X$. Instead we turn to our other family of
metrics $\left\{  \overline{d}_{x_{0}}\right\}  $. In some sense, these
boundary metrics retain more information about the interior space $X$. That
additional information allows us to prove the following theorem, which we view
as a weak solution to Question 1.0.1. It is our primary application of the
$\overline{d}_{x_{0}}$ metrics.

\begin{thmc}
Suppose $X$ is a geodesically complete CAT(0) space and, when endowed with the $\overline{d}_{x_0}$ metric for $x_0\in X$,  $\ell$-$\dim\partial X\leq n$. Then the
macroscopic dimension of $X$ is at most $2n+1$.
\end{thmc}

In Section 5, we compare the $d_{A,x_{0}}$ and $\overline{d}_{x_{0}}$
metrics to each other by applying them to some simple examples. We also
compare them to the established visual metrics when we have a space that is
both CAT(0) and hyperbolic.

Much work remains in this area and thus we conclude with a list of open questions.

\begin{ack} The author would like to thank Craig Guilbault for his guidance and suggestions during the course of this project. \end{ack}

\section{Preliminaries}

Before discussing the possible metrics and their properties, we first review CAT(0) spaces and the visual boundary, quasi-symmetries, and the various dimension theories that will be discussed. The study of metrics on the boundary begins in Section 3. 

\subsection{CAT(0) Spaces and their Geometry}

In this section, we review the definition of CAT(0) spaces, some basic properties of these spaces, and the visual boundary. For a more thorough treatment of CAT(0) spaces, see \cite{BH99}. 

\begin{defn} A geodesic metric space $(X,d)$ is a \textbf{\emph{CAT(0) space}} if all of its geodesic triangles are no fatter than their corresponding Euclidean comparison triangles. That is, if $\Delta(p,q,r)$ is any geodesic triangle in $X$ and $\overline{\Delta}(\overline{p},\overline{q},\overline{r})$ is its comparison triangle in $\EE^2$, then for any $x,y\in\Delta$ and the comparison points $\overline{x},\overline{y}$, then $d(x,y)\leq d_{\EE}(\overline{x},\overline{y})$. 
\end{defn}

A few important properties worth mentioning are that proper CAT(0) spaces are contractible, uniquely geodesic, balls in the space are convex, and the distance function is convex. Furthermore, we now record a very simple geometric property that will be used repeatedly throughout the rest of the paper. 

\begin{lemma} Let $(X,d)$ be a proper CAT(0) space and suppose $\alpha,\beta:[0,\infty)\to X$ are two geodesic rays based at the same point $x_0\in X$. Then for $0<s\leq t<\infty$, $d(\alpha(s),\beta(s))\leq \frac{s}{t}d(\alpha(s),\beta(t))$. 
\end{lemma}

\begin{proof}
Let $p=\alpha(t)$, $q=\beta(t)$, $x=\alpha(s)$, and $y=\beta(s)$. Consider the geodesic triangle $\Delta(x_0, p, q)$ in X and its comparison triangle $\overline{\Delta}(\overline{x_0},\overline{p},\overline{q})$ in $\EE^2$. Let $\overline{x},\overline{y}$ be the corresponding points to $x,y$ on $\overline{\Delta}$. (See picture below.)

\begin{center}
	\includegraphics[scale=0.3]{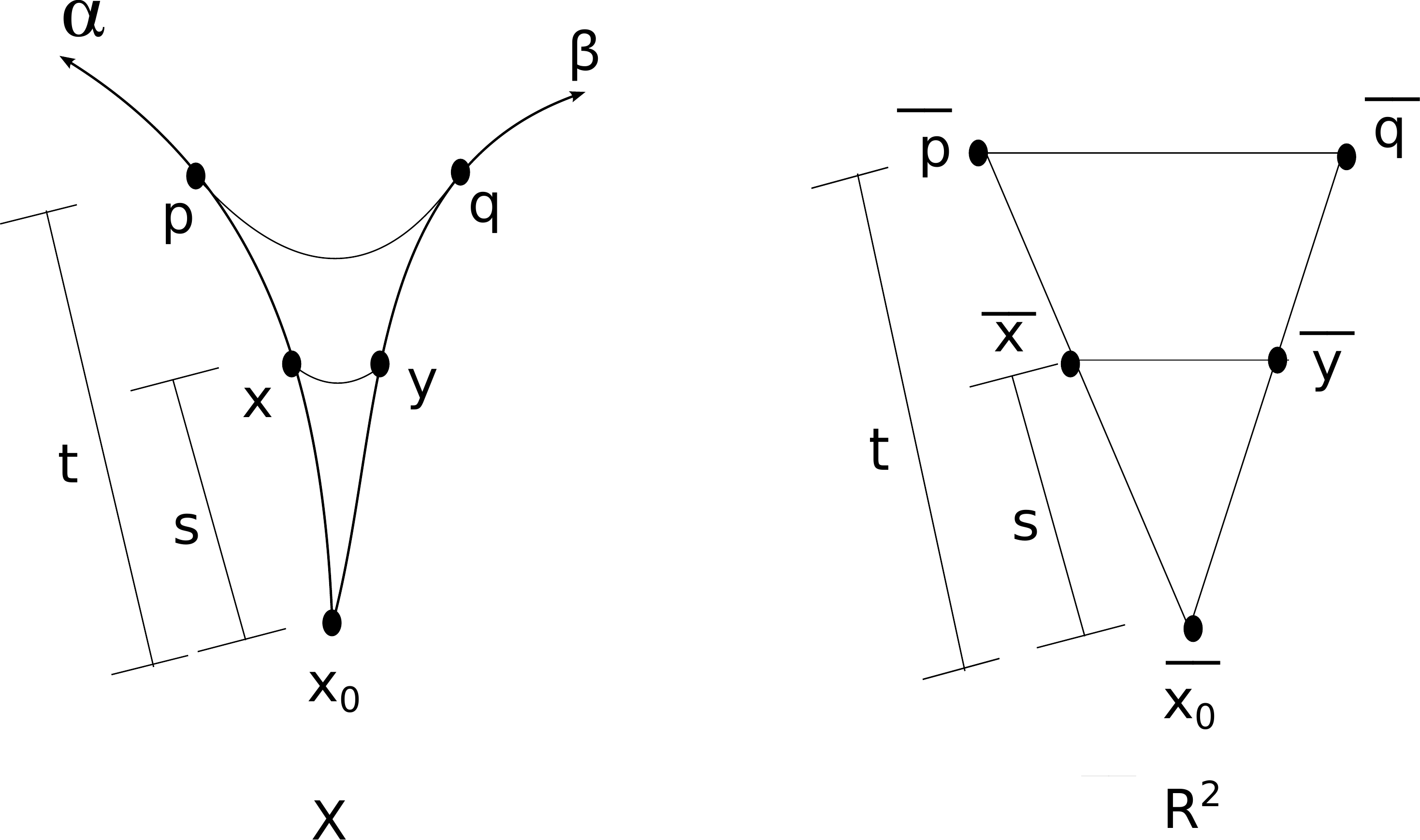}
	\end{center}

By similar triangles in $\EE^2$,
\[ \frac{d_{\EE}(\overline{p},\overline{q})}{d_{\EE}(\overline{x},\overline{y})}=\frac{d_{\EE}(\overline{x_0},\overline{p})}{d_{\EE}(\overline{x_0},\overline{x})}=\frac{t}{s}\]

Thus, $d_{\EE}(\overline{x},\overline{y})=\frac{s}{t}d_{\EE}(\overline{p},\overline{q})=\frac{s}{t}d(p,q)$

Applying the CAT(0)-inequality, we obtain the desired inequality: 
\[d(x,y)\leq \left(\frac{s}{t}\right)d(p,q)\]

\end{proof}

We now review the definition of the  boundary of CAT(0) spaces: 

\begin{defn} The \textbf{\emph{boundary}} of a proper CAT(0) space $X$, denoted $\partial X$, is the set of equivalence classes of rays, where two rays are equivalent if and only if they are asymptotic. We say that two geodesic rays $\alpha, \alpha':[0,\infty)\to X$ are \textbf{\emph{asymptotic}} if there is some constant $k$ such that $d(\alpha(t),\alpha'(t))\leq k$ for every $t\geq 0$. \end{defn}

Once a base point is fixed, there is a a unique representative geodesic ray from each equivalence class by the following:

\begin{prop}[See \cite{BH99} Proposition 8.2] If $X$ is a complete CAT(0) space and $\gamma:[0,\infty)\to X$ is a geodesic ray with $\gamma(0)=x$, then for every $x'\in X$, there is a unique geodesic ray $\gamma':[0,\infty)\to X$ asymptotic to $\gamma$ and with $\gamma'(0)=x'$. \end{prop}

\begin{rk} In the construction of the asymptotic ray for Proposition 2.1.4, it is easy to verify that $d(\gamma(t), \gamma '(t))\leq d(x,x')$ for all $t\geq 0$. 
\end{rk}

We may endow $\overline{X}=X\cup \partial X$, with the \textbf{cone topology}, described below, which makes $\partial X$ a closed subspace of $\overline{X}$ and $\overline{X}$ compact (as long as $X$ is proper). With the topology on $\partial X$ induced by the cone topology on $\overline{X}$, the boundary is often called the \textbf{visual boundary}. In what follows, the term `boundary' will always mean `visual boundary'. Furthermore, we will slightly abuse terminology and call the cone topology restricted to $\partial X$ simply the cone topology if it is clear that we are only interested in the topology on $\partial X$. 

One way in which to describe the cone topology on $\overline {X}$, denoted $\mathscr{T}(x_0)$ for $x_0\in X$, is by giving a basis. A basic neighborhood of a point at infinity has the following form: given a geodesic ray $c$ and positive numbers $r>0$, $\epsilon>0$, let 
\[U(c,r, \epsilon) = \{x \in X | d(x, c(0)) > r, d(p_r(x), c(r)) < \epsilon\}\]
 where $p_r$ is the natural projection of $\overline{X} $ onto $\overline{B}(c(0),r)$.
Then a basis for the topology, $\mathscr{T}(x_0)$, on $\overline{X}$ consists of the set of all open balls $B(x,r) \subset X$, together with the collection of all sets of the
form $U(c,r, \epsilon)$, where $c$ is a geodesic ray with $c(0) = x_0$.

\begin{rk} For all $x_0,x_0'\in X$, $\mathscr{T}(x_0)$ and $\mathscr{T}(x_0')$ are equivalent \cite[Proposition 8.8]{BH99}.\end{rk}
%
%


\subsection{Quasi-Symmetries}

As we are interested in both large-scale and small-scale properties of metric spaces, we briefly discuss two different types of maps that may be used to capture the particular scale we care about. The first type of map is a quasi-isometry. 

\begin{defn} A map $f:(X,d_X)\to(Y,d_Y)$ between metric spaces is a \emph{\textbf{quasi-isometric embedding}} if there exists constants $A,B>0$ such that for every $x,y\in X$, $\frac{1}{A}d_X(x,y)-B\leq d_Y(f(x),f(y))\leq Ad_X(x,y)+B$. Moreover, if there exists a $C>0$ such that for every $z\in Y$, there is some $x\in X$ such that $d_Y(f(x),z)\leq C$, then we call $f$ a \emph{\textbf{quasi-isometry}}. 
\end{defn}

Quasi-isometries capture the large-scale geometry of a metric space, but ignore the small scale-behavior. Thus, they are ideal when studying large scale notions of dimension, which we will discuss briefly in the next section. Since small-scale behavior is ignored, all compact metric spaces turn out to be quasi-isometric because they are all quasi-isometric to a point. Thus, quasi-isometries are not particularly useful when studying compact metric spaces. When interested in compact metric spaces and small-scale behavior, we can turn to a second type of map: quasi-symmetry.

Quasi-symmetric maps were defined to extend the notion of quasi-conformality. Since these maps care about local behavior, they are ideal when studying small scale notions of dimension, in particular, linearly controlled dimension. Quasi-symmetric maps have also played a large role in the the study of hyperbolic group boundaries. For example, it has been shown that all visual metrics on the boundary are quasi-symmetric. 
 
We review the definition and properties that will be needed in later sections. For more information, see \cite{TuVa80} or \cite{He01}. 

\begin{defn} A map $f:X\to Y$ between metric spaces is said to be \textbf{\emph{quasi-symmetric}} if it is not constant and there is a homeomorphism $\eta:[0,\infty)\to[0,\infty)$ such that for any three points $x,y,z\in X$ satisfying $d(x,z)\leq td(y,z)$, it follows that $d(f(x), f(z))\leq\eta(t)d(f(y),f(z))$ for all $t\geq 0$. The function $\eta$ is often called a \textbf{\emph{control function}} of $f$. A \textbf{\emph{quasi-symmetry}} is a quasi-symmetric homeomorphism.
\end{defn}

\begin{thm}\cite[Proposition 10.6]{He01} If $f:X\to Y$ is $\eta$-quasi-symmetric, then $f^{-1}:f(X)\to X$ is $\eta'$-quasi-symmetric where $\eta'(t)=1/\eta^{-1}(t^{-1})$ for $t>0$. Moreover, if $f:X\to Y$ and $g:Y\to Z$ are $\eta_f$ and $\eta_g$ quasi-symmetric, respectively, then $g\circ f:X\to Z$ is $\eta_g\circ\eta_f$ quasi-symmetric. 
 \end{thm}

\begin{thm}\cite[Theorem 11.3]{He01}A quasi-symmetric embedding $f$ of a uniformly perfect space $X$ is $\eta$-quasi-symmetric with $\eta$ of the form $\eta(t)=c*\emph{max}\{t^{\delta},t^{1/\delta}\}$ where $c\geq 1$ and $\delta\in(0,1]$ depends only on $f$ and $X$. 
\end{thm}

We say that a metric space $X$ is \textbf{uniformly perfect} if there exists a $c>1$ such that for all $x\in X$ and for all $r>0$, the set $B(x,r)-B(x,\frac{r}{c})\neq \emptyset$ whenever $X-B(x,r)\neq\emptyset$. Some examples of uniformly perfect spaces include connected spaces and the Cantor ternary set. Being uniformly perfect is a quasi-symmetry invariant \cite{He01}. 


\subsection{A Review of Various Dimension Theories}

Recall that the \textbf{covering dimension} of a space $X$ is at most $n$, denoted dim$X\leq n$, if every open cover of $X$ has an open refinement of order at most $n+1$. The covering dimension can be studied for any topological space, in particular, spaces need not be metrizable. However, if $X$ is a compact metric space, we may use the following to show finite covering dimension. 

\begin{lemma} For a compact metric space $X$, \emph{dim}$X\leq n$ if, for every $\epsilon>0$, there is a cover of $X$ with mesh smaller than $\epsilon$ and order at most $n+1$.\end{lemma}

In the preceding lemma, we use the terms `mesh' and `order'. We now define this terminology, along with a few other terms needed for the other dimension theories.  Given a cover $\mathscr{U}$ of a metric space $X$, we define mesh$(\mathscr{U})=\sup\{\text{diam}(U)|U\in\mathscr{U}\}$. We say that the cover $\mathscr{U}$ is \textbf{uniformly bounded} if there exists some $D>0$ such that mesh$(\mathscr{U})\leq D$. The \textbf{order} of $\mathscr{U}$ is the smallest integer $n$ for which each element $x\in X$ is contained in at most $n$ elements of $\mathscr{U}$. The \textbf{Lebesgue number} of $\mathscr{U}$, denoted $\mathscr{L(U)}$, is defined as $\mathscr{L(U)}=\text{inf}_{x\in X}\mathscr{L}(\mathscr{U},x)$, where $\mathscr{L}(\mathscr{U},x)=\text{sup}\{d(x, X-U)|U\in\mathscr{U}\}$ for each $x\in X$.

%
%

 One reason for pointing out the alternate characterization of covering dimension for compact metric spaces is that the other dimension theories that we discuss here are restricted to metric spaces. These restrictions are due to the need for control of Lebesgue numbers as well as the mesh of covers. In particular, we record two properties for covers that will be used to characterize the different notions of dimension. 
 
 Let $\mathscr{U}$ be a uniformly bounded open cover of a metric space $X$. We say that $\mathscr{U}$ has 
 \begin{itemize}

\item Property $\mathscr{P}_\lambda^n$ if $\mathscr{L(U)}\geq\lambda$ and order$(\mathscr{U})\leq n+1$.
\item Property $\mathscr{P}_{\lambda, c}^n$ if $\mathscr{L(U)}\geq\lambda$, mesh$(\mathscr{U)}\leq c\lambda$, and order$(\mathscr{U})\leq n+1$ 

\end{itemize}
This second property requires not only a given Lebesgue number, but also a linear relationship between the mesh of the cover and the Lebesgue number. These two properties capture key requirements in the remaining dimension theories, which we now describe, organized in terms of large-scale and small-scale properties. 

\begin{defn} Let $X$ be a metric space.

\begin{enumerate}
	\item The \textbf{\emph{macroscopic dimension}} of $X$ is at most $n$, denoted $\dim_{\emph{mc}}X\leq n$, if there exists a single uniformly bounded open cover of $X$ with order  $n+1$. 
	\item The \textbf{\emph{asymptotic dimension}} of $X$ is at most $n$, denoted \emph{asdim}$X\leq n$, if for every $\lambda>0$, there exists a cover $\mathscr{U}$ with Property $\mathscr{P_\lambda}^n$.
	\item The \textbf{\emph{linearly-controlled asymptotic dimension}} of $X$ is at most $n$, denoted $\ell\emph{-asdim}X\leq n$, if there exists $c\geq 1$ and $\lambda_0>0$ such that for all $\lambda\geq \lambda_0$, there is a cover $\mathscr{U}$ with Property $\mathscr{P}_{\lambda,c}^n$. 
	\item The \textbf{\emph{Assouad-Nagata dimension}} of $X$ is at most $n$, denoted \emph{ANdim}$X\leq n$, if there exists $c\geq 1$, such that for all $\lambda>0$, there is a cover $\mathscr{U}$ with Property $\mathscr{P}_{\lambda,c}^n$.
	\item The \textbf{\emph{linearly-controlled dimension}} of $X$ is at most $n$, denoted\\ $\ell\emph{-dim}X\leq n$, if there exists $c\geq 1$ and $\lambda_0>0$ such that for all $0<\lambda\leq\lambda_0$, there is a cover $\mathscr{U}$ with Property $\mathscr{P}_{\lambda,c}^n$.
	\end{enumerate}
\end{defn}

We wish to record a few facts about the various dimension theories, as well as some relationships between them:
\begin{enumerate}
	\item Asymptotic dimension and linearly-controlled asymptotic dimension are quasi-isometry 		invariants of a metric space. For a nice survey of asymptotic dimension, see \cite{BDr07}. It has 	become widely studied due in part to its relationship to the Novikov Conjecture. 
	\item Assouad-Nagata dimension is a quasi-symmetry invariant \cite{LS05}. Since 	$\ell\text{-dim}X= \text{ANdim}X$ for bounded metric spaces,  linearly-controlled  dimension is a quasi-symmetry invariant for bounded metric spaces
	\item In fact, linearly-controlled metric dimension is a quasi-symmetry invariant of a larger class of 	metric spaces: uniformly perfect metric spaces \cite{BuSc07}.
	\item For a metric space $X$, we have the following comparisons:  
		\[\text{mdim}X\leq\text{dim}X\leq\ell\text{-dim}X\leq \text{ANdim}X\]
		\[\text{mdim}X\leq\text{asdim}X\leq\ell\text{-asdim}X\leq\text{ANdim}X\]
\end{enumerate}

For more on the above dimension theories, see \cite{BuSc07}

\section{ The $d_{A,x_0}$ metrics}	

We are now ready to define the first family of metrics on the visual boundary of a CAT(0) space: the $d_{A,x_0}$ metrics.  

Fix a base point $x_0\in X$ and choose $A>0$. For $[\alpha],[\beta]\in\partial X$, let $\alpha:[0,\infty)\to X$ and $\beta:[0,\infty)\to X$ be the geodesic rays based at $x_0$ and asymptotic to $[\alpha] $ and $[\beta]$, respectively. Let $a\in(0,\infty)$ be such that $d(\alpha(a),\beta(a))=A$. If such an $a$ does not exist, set $a=\infty$. Then, define $d_{A,x_0}:\partial X\times\partial X\to \RR$ by 
		\[d_{A,x_0}([\alpha],[\beta])=\frac{1}{a}\]
  

\subsection{Basic Properties of the $d_{A,x_0}$ metrics}

Before discussing any properties of the $d_{A,x_0}$ metrics, we must first show that each member of the family is indeed a metric and induces the cone topology on $\partial X$. 

\begin{lemma} If $(X,d)$ is a CAT(0) space and $x_0\in X$, then $d_{A,x_0}$ for any $A>0$ is a metric on $\partial X$. \end{lemma}

\begin{proof} Fix a base point $x_0\in X$ and choose $A>0$. Let $[\alpha],[\beta],[\gamma]\in\partial X$ and $\alpha, \beta,\gamma: [0,\infty)\to X$ be the geodesic rays based at $x_0$ and asymptotic to $[\alpha],[\beta],[\gamma]$, respectively. 

Clearly, $d_{A,x_0}([\alpha],[\alpha])=0$ since $d(\alpha(t),\alpha(t))=0$ for every $t\geq 0$ and hence $a=\infty$. If $d_{A,x_0}([\alpha],[\beta])=0$, then there is no $a\in (0,\infty)$ such that $d(\alpha(a),\beta(a))=A$. By convexity of CAT(0) metric, this means $d(\alpha(t),\beta(t))=0$ for every $t\geq 0$. Hence, $\alpha=\beta$, which means $[\alpha]=[\beta]$. Also, $d_A([\alpha],[\beta])=d_A([\beta],[\alpha])$ since $d(\alpha(t),\beta(t))=d(\beta(t),\alpha(t))$. Finally, to verify the triangle inequality, suppose $a,b,c\in (0,\infty]$ satisfy 
	\[d_{A,x_0}([\alpha],[\beta])=\frac{1}{a} \, , \, 
	d_{A,x_0}([\beta], [\gamma])=\frac{1}{b} \, , \, 
	d_{A,x_0}([\alpha],[\gamma])=\frac{1}{c}\]
If $c\geq a $ or $c\geq b$, then 
\[ d_{A,x_0}([\alpha],[\gamma])=\frac{1}{c}\leq \frac{1}{a}\leq \frac{1}{a}+\frac{1}{b}=d_{A,x_0}([\alpha],[\beta])+d_{A,x_0}([\beta], [\gamma])\]
or 
\[d_{A,x_0}([\alpha],[\gamma])=\frac{1}{c}\leq \frac{1}{b}\leq \frac{1}{a}+\frac{1}{b}=d_{A,x_0}([\alpha],[\beta])+d_{A,x_0}([\beta], [\gamma])\]
Thus, the only interesting case is if $c<a$ and $c<b$. By Lemma 2.1.2 
\[d(\alpha(c),\beta(c))\leq\frac{c}{a}A\]
and 
\[d(\beta(c), \gamma(c))\leq\frac{c}{b}A\]
Then, 
\[A=d(\alpha(c),\gamma(c))\leq d(\alpha(c), \beta(c))+d(\beta(c),\gamma(c))\leq \frac{c}{a}A+\frac{c}{b}A=Ac\left(\frac{a+b}{ab}\right)\]

Thus, 
\[c\geq \frac{ab}{a+b}\]
which proves:
\[d_{A,x_0}([\alpha],[\gamma])=\frac{1}{c}\leq\frac{a+b}{ab}=\frac{1}{a}+\frac{1}{b}=d_{A,x_0}([\alpha],[\beta])+d_{A,x_0}([\beta], [\gamma])\]

\end{proof}

\begin{lemma} The topology induced by the $d_{A,x_0}$ metric on $\partial X$ is equivalent to the cone topology on $\partial X$.
\end{lemma} 

\begin{proof} Fix $A>0$ and $x_0\in X$. Since the base point is fixed, we will simplify $d_{A,x_0}$ to $d_A$. Consider the basic open set $B_{d_A}([\alpha],\epsilon)$ for $[\alpha]\in\partial X$ and $\epsilon>0$ and let $[\beta]\in B_{d_A}([\alpha],\epsilon)$. Let $\alpha,\beta:[0,\infty)\to X$ be the unique geodesic rays based at $x_0$ corresponding to $[\alpha]$ and $[\beta]$, respectively. Choose $\delta>0$ such that $B_{d_A}([\beta],\delta)\subset B_{d_A}([\alpha],\epsilon)$ and consider the basic open set in the cone topology $U(\beta,\frac{1}{\delta},A)\cap\partial X$. Let $[\gamma]\in U(\beta,\frac{1}{\delta},A)\cap\partial X$. Then $d(\beta(\frac{1}{\delta}),\gamma(\frac{1}{\delta}))<A$. If $a>0$ is the point such that $d(\beta(a),\gamma(a))=A$, then $a>\frac{1}{\delta}$. Thus, $d_A([\beta],[\gamma])=\frac{1}{a}<\delta$. Thus, $[\gamma]\in B_{d_A}([\beta],\delta)\subset B_{d_A}([\alpha],\epsilon)$, proving $[\beta]\in U(\beta,r,A)\cap \partial X \subset B_{d_A}([\alpha],\epsilon)$. 

Now consider a basic open set $U(\alpha, r, \epsilon)\cap\partial X$ in the cone topology where $r>0$, \\$A>\epsilon>0$ and $\alpha:[0,\infty)\to X$ is a geodesic ray based at $x_0$ . Let $[\beta]\in U(\alpha, r, \epsilon)\cap\partial X$. Choose $\delta>0$ such that $B_d(\beta(r),\delta)\cap S(x_0,r)\subset B_d(\alpha(r),\epsilon)\cap S(x_0,r)$ and consider the basic open set in the metric topology $B_{d_A}([\beta],\frac{\delta}{Ar})$. Let $[\gamma]\in B_{d_A}([\beta],\frac{\delta}{Ar})$. Then $d_A([\beta],[\gamma])=\frac{1}{a}<\frac{\delta}{Ar}$ where $a>0$ is such that $d(\beta(a),\gamma(a))=A$, which means $a>r$ since $A>\epsilon\geq \delta$. By Lemma 2.1.2, $d(\gamma(r),\beta(r))\leq \frac{r}{a}A<r\frac{\delta}{Ar}A=\delta$. Thus, $\gamma(r)\in B_d(\beta(r),\delta)\cap S(x_0,r)\subset B_d(\alpha(r),\epsilon)\cap S(x_0,r)$, proving $[\gamma]\in U(\alpha,r,\epsilon)$. Thus $[\beta]\in B_{d_A}([\beta],\frac{\delta}{Ar})\subset U(\alpha,r,\epsilon)$. 
\end{proof} 

\begin{rk} Recall that the cone topology is defined on $\overline{X}=X\cup\partial X$. However, the preceding lemma restricts the cone topology to the boundary since there is not a natural extension of $d_{A,x_0}$ to $\overline{X}$.  
\end{rk}

We now answer two important questions: what happens if we change $A$ and what happens if we move the base point? It turns out that in both cases, the metrics are quasi-symmetric. Thus, by transitivity, all members of the $d_{A,x_0}$ family are quasi-symmetric.

\begin{lemma} Let $X$ be a proper CAT(0)-space. For all  $A, A'>0$, $ id_{\partial X}:(\partial X, d_{A,x_0})\to(\partial X, d_{{A',x_0}})$ is a quasi-symmetry. 
\end{lemma}

\begin{proof} Fix a base point $x_0\in X$ and suppose, without loss of generality, that $A<A'$.  Clearly the identity map is a homeomorphism, so we need only verify that $id_{\partial X}$ is a quasi-symmetric map. Let $\eta(t)=\frac{A'}{A}t$; we will show this a control function for $id_{\partial X}$. Suppose that $[\alpha],[\beta],[\gamma]\in \partial X$ with $d_{A,x_0}([\alpha],[\gamma])\leq d_{A,x_0}([\beta],[\gamma])$ for $t>0$. Let $\alpha, \beta,\gamma:[0,\infty)\to X$ be geodesic rays based at $x_0$ that are asymptotic to $[\alpha],[\beta],[\gamma]$, respectively. Let $a,b,a',b'>0$ be such that 
\[d_{A,x_0}([\alpha],[\gamma])=\frac{1}{a} \, , \, d_{A,x_0}([\beta],[\gamma])=\frac{1}{b}\]
\[d_{A',x_0}([\alpha],[\gamma])=\frac{1}{a'}\, , \, d_{A',x_0}([\beta],[\gamma])=\frac{1}{b'}\]

By convexity of CAT(0) metric and since $A'>A$, then $a\leq a'$ and $b\leq b'$. Furthermore, applying Lemma 2.1.2, 
\[A=d(\beta(b),\gamma(b))\leq d_{\EE}(\overline{\beta(b)},\overline{\gamma(b)})=\frac{A'b}{b'}\]

Thus, $\frac{Ab'}{A'}\leq b$. Applying the above,  we obtain the following inequalities:

\[d_{A',x_0}([\alpha],[\gamma])=\frac{1}{a'}\leq\frac{1}{a}=d_{A,x_0}([\alpha],[\gamma])\leq td_{A,x_0}([\beta],[\gamma])=t\frac{1}{b}\leq t\frac{A'}{A}\frac{1}{b'}=\eta(t)d_{A',x_0}([\beta],[\gamma])\]

\end{proof}



\begin{lemma} Suppose $X$ is a complete CAT(0) space. For all $x_0,x_0'\in X$, \\$id_{\partial X}:(\partial X, d_{A,x_0})\to(\partial X, d_{A,x_0'})$ is a quasi-symmetry. 
\end{lemma}

\begin{proof} Let $x_0,x_0'\in X$ with $x_0\neq x_0'$. We begin by assuming $A>2d(x_0,x_0')$. We show that $\eta(t)=\left(\frac{A}{A-2d(x_0,x_0')}\right)^2t$ is a control function for $id_{\partial X}$. Suppose that $[\alpha],[\beta],[\gamma]\in\partial X$ and satisfy the inequality $d_{A,x_0}([\alpha],[\gamma])\leq td_{A,x_0}([\beta],[\gamma])$ for $t>0$. Let $\alpha,\beta,\gamma:[0,\infty)\to X$ be geodesic rays based at $x_0$ and asymptotic to the corresponding points in $\partial X$. Let $a,b\in (0,\infty)$ be such that $d_{A,x_0}(\alpha(a),\gamma(a))=A$ and $d_{A,x_0}(\beta(b),\gamma(b))=A$. 

Since $X$ is a complete CAT(0) space, there exists unique geodesic rays $\alpha',\beta',\gamma'$ in $X$ based at $x_0'$ and asymptotic to $\alpha, \beta, \gamma$, respectively. Let $a',b'\in(0,\infty)$ be such that $d_{A,x_0'}(\alpha'(a'),\gamma'(a'))=A$ and $d_{A,x_0'}(\beta'(b'),\gamma'(b'))=A$. There are four cases to consider:

\noindent \underline{Case 1}: $a'\geq a$ and $b\geq b'$. Then
\[d_{A,x_0'}([\alpha],[\gamma])=\frac{1}{a'}\leq\frac{1}{a}=d_{A,x_0}([\alpha],[\gamma])\leq td_{A,x_0}([\beta], [\gamma])=t\frac{1}{b}\]\[\leq t\frac{1}{b'}=td_{A,x_0'}([\beta],[\gamma])\leq\eta (t)d_{A,x_0'}([\beta],[\gamma])\]

\noindent \underline{Case 2}: $a'\geq a$ and $b<b'$. Applying Lemma 2.1.2, $d(\beta'(b),\gamma'(b))\leq\frac{Ab}{b'}$. Thus, $\frac{b'}{A}d(\beta'(b),\gamma'(b))\leq b$. Furthermore, by Remark 1,
\[A=d(\beta(b),\gamma(b))\leq d(\beta(b),\beta'(b))+d(\beta'(b),\gamma'(b))+d(\gamma'(b),\gamma(b))\leq 2d(x_0,x_0')+d(\beta'(b),\gamma'(b))\]

Thus, $A-2d(x_0,x_0')\leq d(\beta'(b),\gamma'(b))$

Applying all of the above,

\[d_{A,x_0'}([\alpha],[\gamma])=\frac{1}{a'}\leq \frac{1}{a}=d_{A,x_0}([\alpha],[\gamma])\leq td_{A,x_0}([\beta], [\gamma])=t\frac{1}{b}\]\[\leq t\frac{A}{d(\beta'(b),\gamma'(b))}\frac{1}{b'}\leq t\frac{A}{A-2d(x_0,x_0')}d_{A,x_0'}([\beta],[\gamma])\leq \eta(t)d_{A,x_0'}([\beta],[\gamma])\]

\noindent \underline{Case 3}: $a'<a$ and $b\geq b'$ Using Lemma 2.1.2, $d(\alpha(a'),\gamma(a'))\leq \frac{Aa'}{a}$. Furthermore, by Remark 1, 
\[A=d(\alpha'(a'),\gamma'(a'))\leq d(\alpha'(a'),\alpha(a'))+d(\alpha(a'),\gamma(a'))+d(\gamma(a'),\gamma'(a'))\leq 2d(x_0,x_0')+d(\alpha(a'),\gamma(a'))\]

Applying the above, 

\[d_{A,x_0'}([\alpha],[\gamma])=\frac{1}{a'}\leq \frac{A}{d(\alpha(a'),\gamma(a'))}\frac{1}{a}\leq \frac{A}{A-2d(x_0,x_0')}\frac{1}{a}=\frac{A}{A-2d(x_0,x_0')}d_{A,x_0}([\alpha],[\gamma])\]\[\leq \frac{A}{A-2d(x_0,x_0')}td_{A,x_0}([\beta], [\gamma])=\frac{A}{A-2d(x_0,x_0')}t\frac{1}{b}\leq \frac{A}{A-2d(x_0,x_0')}t\frac{1}{b'}\]\[=\frac{A}{A-2d(x_0,x_0')}td_{A,x_0'}([\beta], [\gamma])\leq \eta(t)d_{A,x_0'}([\beta],[\gamma])\]

\noindent \underline{Case 4}: $a'<a$ and $b<b'$. Using the computations in Cases 2 and 3:
\[d_{A,x_0'}([\alpha],[\gamma])=\frac{1}{a'}\leq \frac{A}{A-2d(x_0,x_0')}d_{A,x_0}([\alpha],[\gamma])\]\[\leq \frac{A}{A-2d(x_0,x_0')}td_{A,x_0}([\beta], [\gamma])=\frac{A}{A-2d(x_0,x_0')}t\frac{1}{b}\leq t\left(\frac{A}{A-2d(x_0,x_0')}\right)^2\frac{1}{b'}\]\[=t\left(\frac{A}{A-2d(x_0,x_0')}\right)^2d_{A,x_0'}([\beta],[\gamma]) = \eta(t)d_{A,x_0'}([\beta],[\gamma])\]

Thus, $\eta(t)=\left(\frac{A}{A-2d(x_0,x_0')}\right)^2t$ is a control function for $id_{\partial X}$ for $A>2d(x_0,x_0')$.

Now, suppose we are given any $A>0$. Since $X$ is a CAT(0) space, it is path connected. Let $\gamma:[0,d(x_0,x_0')]\to X$ be a geodesic segment connecting $x_0$ to $x_0'$. Let $\{y_0,y_1,...,y_{n-1},y_n\}$ be a partition of $[0,d(x_0,x_0')]$ where $|x_k-x_{k-1}|<\frac{A}{2}$ for $k=1,2,...n$ and set $x_k=\gamma(y_k)$ for $k=0,1,...,n-1$ and $x_0'=\gamma(y_n)$. From above, we know $id^k_{\partial X}:(\partial X,d_{A,x_k})\to(\partial X, d_{A,x_{k-1}})$ is a quasi-symmetry for each $k$. Theorem 2.2.3 guarantees that $id_{\partial X}=id^n_{\partial X}\circ...\circ id^1_{\partial X}:(\partial X, d_{A,x_0})\to(\partial X, d_{A,x_0'})$ is a quasi-symmetry. 
\end{proof}

%
In the future, we will use $d_A$ to denote an arbitrary representative of the family of metrics $\{d_{A,x_0}\}$. When specific calculations are to be done, $A>0$ should be fixed and a base point $x_0$ should be chosen. 

In problem 46 of \cite{Ka07}, B. Kleiner asked whether the group of isometries of a CAT(0) space acts in a ``nice'' way on the boundary.  The following theorem provides one answer. 

\begin{thm} Suppose $G$ is a finitely generated group that acts by isometries on a complete CAT(0) space $X$. Then the induced action of $G$ on $(\partial X, d_{A,x_0})$ is a quasi-symmetry. In other words, $G$ acts by quasi-symmetries on $\partial X$.
\end{thm}

\begin{proof} Fix a base point $x_0\in X$ and $A>0$. Notice that proving this theorem relies on knowing that changing base point is a quasi-symmetry, since if $\alpha,\beta,\gamma:[0,\infty)\to X$ are geodesic rays based at $x_0$, then 
\[d_{A,x_0}([\alpha],[\gamma])=d_{A,gx_0}([g\alpha],[g\gamma])\]
\[d_{A,x_0}([\beta],[\gamma])=d_{A,gx_0}([g\beta],[g\gamma]).\]
This is a simple consequence of the action being by isometries. Hence, to obtain the desired inequality for a quasi-symmetric map, all we need to do is find the distances of the translated rays with respect to the base point $x_0$ rather than $gx_0$. A simple application of Theorem 3.1.4 proves $g$ is a quasi-symmetry. 

%
\end{proof}

\subsection {Dimension Results Using the $d_A$ metric}

In \cite{BuLe07}, it is shown that the linearly controlled dimension of every compact
locally self-similar metric space $X$ is finite and $\ell\hbox{-dim}X=\hbox{dim}X$. Since hyperbolic group boundaries are compact and locally self-similar, we obtain the equality of linearly controlled dimension and covering dimension of hyperbolic group boundaries in Theorem 1.0.2. Swenson shows in \cite{Swe99} that the boundary of a proper CAT(0) space admitting a cocompact action by isometries has finite topological dimension. Since topological dimension can be defined for arbitrary topological spaces, there was no need for a metric on the boundary to prove this fact. Now that we have the $d_A$ family of metrics on the boundary, we can examine the linearly controlled metric dimension. We have been unable to show equality of the two dimensions, but we do show that linearly controlled dimension of a CAT(0) group boundary must be finite. This proof was motivated by previous work found in \cite{Mo14}.

\begin{thm} Suppose $G$ acts geometrically on a proper CAT(0)-space $X$. Then $\ell$-\emph{dim}$(\partial X, d_A)<\infty$.
\end{thm}

This proof relies on the existence of a single cover with Property $\mathscr{P}_{R,4R}^n$ for some $R,n>0$. 

\begin{lemma} Suppose a group $G$ acts geometrically on a proper CAT(0) space $(X,d)$. Then for all sufficiently large $R$, there exists a finite order open cover $\mathscr{V}$ of $X$ with \emph{mesh}$(\mathscr{V})\leq 4R$ and $\mathscr{L(V)}\geq R$. 
\end{lemma}

\begin{proof} Let $C\subseteq X$ be a compact set with $GC=X$ and choose $R$ large enough so that $C\subseteq B(x_0,R)$ for some $x_0\in X$. Then $\mathscr{V}=\cup_{g\in G}B(gx_0,2R)$ is a finite order open cover of $X$ with mesh bounded above by $4R$. Notice that the order of $\mathscr{V}$ is finite since the action of $G$ is proper, that is only finitely many $G$-translates of any compact set $C$ can intersect $C$. Since the cover is obtained by this nice geometric action, it must look the same everywhere. Thus, the order of $\mathscr{V}$ is bounded above by the finite number of translates of $gB(x_0,2R)$ intersecting $B(x_0,2R)$.  Furthermore, the Lebesgue number of $\mathscr{V}$ is at least $R$. For if we take $x\in X$ and let $g\in G$ such that $gx\in C\subseteq B(x_0,R)$. Then $d(gx, X-B(x_0,2R))\geq R$. As the action is by isometries:
\[R\leq d(gx, X-B(x_0,2r))=d(x, g^{-1}(X-B(x_0,2R)))=d(x, X-g^{-1}(B(x_0, 2R)))\]\[=d(x, X-B(g^{-1}x_0,2R))\]

Since $B(g^{-1}x_0,2R)\in\mathscr{V}$, and $d(x,B(g^{-1}x_0,2R))\geq R$, then $\mathscr{L}(\mathscr{V})\geq R$. 
\end{proof}

\begin{rk} Lemma 3.2.2 proves that $\dim_{\emph{mc}} X<\infty$ for a CAT(0) space admitting a geometric action.\end{rk}

\begin{proof} [Proof of Theorem 3.2.1] Fix $A>0$. By Lemma 3.2.2, we may choose a sufficiently large $R>A$ so that there is a finite order open cover $\mathscr{V}$ of $X$ with mesh$(\mathscr{V})\leq 4R$ and $\mathscr{L(V)}\geq R$. Set $n= $order$(\mathscr{V})$. 

\noindent Set $t_{\lambda}=\frac{1}{\lambda}$  for each $\lambda\in(0, \infty)$, and for each $V\in\mathscr{V}$, define
\[U_V=\{[\gamma] | \gamma \text{ is a geodesic ray based at } x_0 \text{ with } \gamma(t_{\lambda})\in V\}\]

We will show that $\mathscr{U}=\cup_{V\in\mathscr{V}}U_V$ is an open cover of $\partial X$ with order bounded above by $n$, Lebesgue number at least $\lambda$ and mesh at most $\frac{4R}{A}\lambda$. 

Clearly $\mathscr{U}$ is an open cover since $\mathscr{V}$ is an open cover of $X$. Furthermore, since $\gamma(t_{\lambda})$ can be in at most $n$-elements of $\mathscr{V}$, then $[\gamma]$ can be in at most $n$ elements of $\mathscr{U}$. 

We now show the Lebesgue number must be at least $\lambda$. Let $[\gamma]\in\partial X$ and $\gamma$ a geodesic ray in $X$ based at $x_0$ and asymptotic to $[\gamma]$. Since $\mathscr{L(V)}\geq R$, there is some $V\in\mathscr{V}$ such that $d(\gamma(t_{\lambda}), X-V)\geq R$. Consider then $d_A([\gamma],\partial X-U_V)$. If $[\beta]\in \partial X-U_V$, then $\beta(t_{\lambda})\notin V$ and hence $d(\gamma(t_{\lambda}),\beta(t_{\lambda}))\geq R$. Letting $a\in (0,\infty)$ be such that $d(\gamma(a), \beta(a))=A$, then $a\leq t_{\lambda}$ since $R\geq A$. Hence, 
\[ d_A([\gamma],[\beta])=\frac{1}{a}\geq \frac{1}{t_{\lambda}}=\lambda\]
Hence, $d_A([\gamma],\partial X-U_V)\geq \lambda$, so $\mathscr{L}(\mathscr{V})\geq \lambda$.

Lastly, we show mesh$(\mathscr{U})\leq \frac{4R}{A}\lambda$. Let $[\alpha],[\beta]\in U_V$ for some $U_V\in\mathscr{U}$. Let $\alpha,\beta$ be geodesic rays in $X$ based at $x_0$ and asymptotic to $[\alpha]$ and $[\beta]$, respectively. Let $a\in(0,\infty)$ be such that $d(\alpha(a), \beta(a))=A$. Since $\alpha(t_{\lambda}),\beta(t_{\lambda})\in V$, then $d(\alpha(t_{\lambda}),\beta(t_{\lambda}))\leq 4R$. There are then two cases to consider:

\noindent \underline{Case 1}: $d(\alpha(t_{\lambda}),\beta(t_{\lambda}))\leq A$. Then $a\geq t_{\lambda}$, so 
\[d_A([\alpha],[\beta])=\frac{1}{a}\leq\frac{1}{t_{\lambda}}=\lambda\leq \frac{4R}{A}\lambda\]

\noindent \underline{Case 2}: $A\leq d(\alpha(t_{\lambda}),\beta(t_{\lambda}))\leq 4R$. Then $a\leq t_{\lambda}$, and by Lemma 2.1.2, $d(\alpha(a),\beta(a))\leq \frac{a}{t_{\lambda}}d(\alpha(t_{\lambda}),\beta(t_{\lambda}))$. Thus, 
\[A=d(\alpha(a),\beta(a))\leq \frac{a}{t_{\lambda}}d(\alpha(t_{\lambda}),\beta(t_{\lambda}))\leq \frac{a}{t_{\lambda}}(4R)\]

Rearranging, we obtain that $a\geq \frac{At_{\lambda}}{4R}$, and thus:
\[d_A([\alpha],[\beta])=\frac{1}{a}\leq\frac{4R}{At_{\lambda}}=\frac{4R}{A}\lambda\]

Thus, there exists a $c\geq 1$ such that for every $\lambda>0$, there is an open cover $\mathscr{U}$ of $\partial X$ with order$(\mathscr{U})\leq n$, $\mathscr{L}(\mathscr{U})\geq\lambda$ and mesh$(\mathscr{U})\leq c\lambda$, proving $\ell$-dim$(\partial X, d_A)<\infty$. 

\end{proof}

The above proof really only required the existence of a single finite order uniformly bounded open cover with large Lebesgue number. Thus, if we know a proper CAT(0) space has finite asymptotic dimension, we do not need a group action to provide such a cover. We point out that there are some CAT(0) spaces that are known to have finite asymptotic dimension: $\RR^n$ for all $n\geq 0$, Gromov hyperbolic CAT(0) spaces, and CAT(0) cube complexes \cite{WR12}. Thus, there are spaces for which the following proposition will apply. 

\begin{prop} Suppose $(X,d)$ is a proper CAT(0) space with finite asymptotic dimension. Then $\ell$-\emph{dim}$(\partial X,d_A)\leq$ \emph{asdim}$X$. 
\end{prop} 

\begin{proof} Fix $A>0$. Since asdim$X\leq n$ for some $n>0$, there exists a uniformly bounded cover $\mathscr{V}$ with order $\mathscr{V}\leq n+1$ and $\mathscr{L(V)}\geq R$ for some $R\geq A$. We may assume that this cover is also open, because if it is not, we can simply choose a larger $R$, ``push in'' the cover $\mathscr{V}$ using the Lebesgue number, and obtain a smaller open cover with the desired properties. Repeat the same argument as in the proof of Theorem 3.2.2 to obtain an open cover $\mathscr{U}$ of $\partial X$ with order at most $n+1$, $\mathscr{L(U)}\geq \lambda$ and mesh$\mathscr{U}\leq \frac{\text{mesh}\mathscr{V}}{A}\lambda$. 
\end{proof} 

%


\section {The $\overline{d}_{x_0}$-metrics}

To define the second family of metrics on $\partial X$, fix a base point $x_0\in X$. For $[\alpha],[\beta]\in\partial X$, let $\alpha:[0,\infty)\to X$ and $\beta:[0,\infty)\to X$ be the unique representatives of $[\alpha]$ and $[\beta]$ based at $x_0$. Define $\overline{d}_{x_0}:\partial X\times\partial X\to \RR$ by 
		
		\[\overline{d}_{x_0}([\alpha],[\beta])=\int_0^{\infty}\frac{d(\alpha(r),\beta(r))}{e^r}\, dr\]

This family of metrics, unlike the $d_{A}$ metrics, takes into account the entire timespan of the geodesic rays. Due to this fact, it can naturally be extended to $\overline{X}=X\cup\partial X$. To do so, consider $x,y\in X$. Let $c_x:[0,d(x_0,x)]\to X$ be the geodesic from $x_0$ to $x$ and $c_y:[0,d(x_0,y)]\to X$ the geodesic segment from $x_0$ to $y$. Extend $c_x$ to $c'_x:[0,\infty)\to X$ by letting $c'_x(r)=x$ for all $r>d(x_0,x)$ and $c'(r)=c(r)$ otherwise. Extend $c_y$ to $c'_y:[0,\infty)\to X$ in a similar fashion. Then 
		\[\overline{d}_{x_0}(x,y)=\int_0^\infty\frac{d(c'_x(r),c'_y(r))}{e^r}\, dr\] 

\subsection{Basic Properties of the $\overline{d}_{x_0}$ metrics}

The following lemma that $\overline{d}_{x_0}$ is a metric is trivial. 

\begin{lemma} If $(X,d)$ is a proper CAT(0) space and $x_0\in X$, then $\overline{d}_{x_0}$ is a metric on $\partial X$. \end{lemma}

\begin{lemma} The topology induced on $\overline{X}=X\cup \partial X$ by the $\overline{d}_{x_0}$ metric is equivalent to the cone topology on $\overline{X}$. 
\end{lemma}

\begin{proof} Fix $x_0\in X$. We will denote $\overline{d}_{x_0}$ by $\overline{d}$. We first show the cone topology is finer than the metric topology by considering points in $X$ and $\partial X$, respectively. 

Let $y\in X$ and $B_{\overline{d}}(x,\epsilon)$ be a basic open set in $\overline{X}$ containing $y$ for some $\epsilon>0$ and $x\in \overline{X}$. Choose $\delta>0$ such that $B_{\overline{d}}(y,\delta)\subset B_{\overline{d}}(x,\epsilon)$ and $B_{\overline{d}}(y,\delta)\cap \partial X=\emptyset$. Consider the basic open set $B_d(y,\delta)$ in the cone topology. Clearly, $y\in B_d(y,\delta)$ and if $z\in B_d(y,\delta)$, then $z\in B_{\overline{d}}(y,\delta)$ since $\overline{d}(y,z)<d(y,z)$. Thus, 
\[y\in B_d(y,\delta)\subset B_{\overline{d}}(y,\delta)\subset B_{\overline{d}}(x,\epsilon)\]
 
 Now, let $[\beta]\in\partial X$, and consider the basic open set $B_{\overline{d}}(x,\epsilon)$ for $\epsilon>0$ and $x\in\overline{X}$. Choose $\delta>0$ such that $B_{\overline{d}}([\beta],\delta)\subset B_{\overline{d}}(x,\epsilon)$. Let $t>0$ be such that $e^{-t}<\delta/4$ and consider the basic open set $U(\beta,t,\frac{\delta}{2})$ in the cone topology. Clearly $[\beta]\in U(\beta,t,\frac{\delta}{2})$, so if $[\gamma]\in U(\beta,t,\frac{\delta}{2})\cap\partial X$, then 
\[\overline{d}([\beta],[\gamma])=\int_0^t\frac{d(\beta(r),\gamma(r))}{e^r}\, dr+\int_t^\infty\frac{d(\beta(r),\gamma(r))}{e^r}\, dr\]
\[\leq \int_0^t\frac{d(\gamma(t),\beta(t))}{e^r}\, dr+\int_t^\infty\frac{2(r-t)+d(\gamma(t),\beta(t))}{e^r}\, dr\]
\[=d(\gamma(t),\beta(t))+\frac{2}{e^t}\]
\[<\frac{\delta}{2}+\frac{\delta}{2}=\delta\]
Moreover, if $y\in U(\beta,t,\frac{\delta}{2})\cap X$ and $c_y:[0,d(x_0,y)]\to X$ is the geodesic from $x_0$ to $y$, then 
\[\overline{d}([\beta], x)=\int_0^t\frac{d(c_y(r),\beta(r))}{e^r}\, dr+\int_t^\infty\frac{d(c_y(r),\beta(r))}{e^r}\, dr\]
\[<\int_0^t\frac{d(c_y(t),\beta(t))}{e^r}\, dr+\int_t^\infty\frac{(r-t)+d(c_y(t),\beta(t))}{e^r}\, dr\]
\[<\frac{3\delta}{4}<\delta\]
These two calculations show $U(\beta,t,\frac{\delta}{2})\subset B_{\overline{d}}([\beta,\delta])$ and thus, 
\[[\beta]\in U(\beta,t,\frac{\delta}{2})\subset B_{\overline{d}}([\beta,\delta])\subset B_{\overline{d}}(x,\epsilon)\]

Now, we show the metric topology is finer than the cone topology, again by considering points in $X$ and $\partial X$. 

Let $y\in X$ and $B$ a basic open set in the cone topology. Choose $\delta>0$ such that $B_d(y,\delta)\subset B\cap X$. Consider the basic open set $B_{\overline{d}}(y,R)$ where $R=\frac{\delta}{e^{d(x_0,y)}}$ (if necessary, choose $R$ smaller so that $B_{\overline{d}}(y,R)\subset X$). Let $z\in B_{\overline{d}}(y,R)$ and $c_y$ and $c_z$ the geodesics connecting $x_0$ to $y$ and $z$, respectively. Set $t=\hbox{max}\{d(x_0,y),d(x_0,z)\}$. Then 
\[\overline{d}(y,z)>\int_t^\infty\frac{d(c_z(r),c_y(r))}{e^r}\, dr\]
\[=\int_t^\infty\frac{d(y,z)}{e^r}\, dr\]
\[=\frac{d(y,z)}{e^t}\]
\[\geq \frac{d(y,z)}{e^{d(x_0,y)}}\]
Since $\overline{d}(y,z)<\frac{\delta}{e^{d(x_0,y)}}$, by the above calculation, $z\in B_d(y,\delta)$ proving 
\[y\in B_{\overline{d}}(y,R)\subset B_d(y,\delta)\subset B\] 

For a boundary point $[\beta]\in \partial X$, let $U(\alpha,t,\epsilon)$ be a basic open set containing $[\beta]$ for $t,\epsilon>0$ and $\alpha$ a geodesic ray based at $x_0$. Choose $1>\delta>0$ so that \\$B_d(\beta(t), \delta)\cap S(x_0, t)\subset B_d(\alpha(t),\epsilon)\cap S(x_0,t)$. Consider the basic open set $B_{\overline{d}}([\beta],\frac{\delta}{e^t})$. If $[\gamma]\in B_{\overline{d}}([\beta],\frac{\delta}{e^t})\cap\partial X$, then $d(\beta(t),\gamma(t))<\delta$. Otherwise, 
\[\overline{d}([\gamma],[\beta])\geq\int_t^\infty\frac{\delta}{e^r}\, dr=\frac{\delta}{e^t}\]
Thus, $d(\gamma(t), \beta(t))<\delta<\epsilon$, so $[\gamma]\in U([\alpha],t,\epsilon)$. If $x\in B_{\overline{d}}([\beta],\frac{\delta}{e^t})\cap X$, we first notice that $\overline{d}(x,[\beta])\geq \overline{d}([\beta],\beta(d(x_0,x)))=e^{-d(x_0,x)}$. Thus, $d(x_0, x)\geq t$, otherwise $x\notin B_{\overline{d}}([\beta],\frac{\delta}{e^t})$. By the same argument just given for a boundary point, we see that $d(c_x(t),\beta(t))<\delta$ proving $x\in U([\alpha],t,\epsilon)$. Thus, 
\[[\beta]\in B_{\overline{d}}\left([\beta],\frac{\delta}{e^t}\right)\subset U([\alpha],t,\epsilon)\]
\end{proof} 

Thus far, we have been unable to prove analogs of Lemma 3.1.4 and Theorem 3.1.5 for this family of metrics. However, we will see that there are some significant advantages in using $\overline{d}_{x_0}$ for comparing dimension properties of $\partial X$ and $X$. In particular, we use the $\overline{d}_{x_0}$ metric to obtain a weak solution to Question 1.0.1 (which we have been unable to accomplish using the $d_A$ metrics).


\subsection{Dimension Results Using the $\overline{d}_{x_0}$ Metrics} 

\begin{thm} Suppose $X$ is a geodesically complete CAT(0) space and $\ell$-\emph{dim}$\partial X\leq n$, where $\partial X$ is endowed with the $\overline{d}_{x_0}$ metric. Then the macroscopic dimension of $X$ is bounded above by $2n+1$. \end{thm} 

The proof ``pushes in'' covers of the boundary obtained by knowing finite linearly controlled metric dimension of the boundary to create covers of the entire space. 

\begin{proof}[Proof of Theorem 4.2.1] We will show that there exists a uniformly bounded cover $\mathscr{V}$ of $X$ with order$\mathscr{V}\leq 2n+1$. Fix a base point $x_0\in X$. Since $\ell-$dim$\partial X\leq n$, there exists constants $\lambda_0\in(0,1)$ and $c\geq 1$ and $n+1$-colored coverings (by a single coloring set $A$) $\mathscr{U}_k$ of $\partial X$ with 
	\begin{itemize}
		\item mesh$\mathscr{U}_k\leq c\lambda_k$
		\item $\mathscr{L}(\mathscr{U}_k)\geq\lambda_k/2$
		\item $\mathscr{U}_k^a$ is $\lambda_k/2$-disjoint for each $a\in A$. 
	\end{itemize}
where $\lambda_k\leq\lambda_0$. Such a cover is guaranteed by \cite[Lemma 11.1.3]{BuSc07}. 

Choose $R>0$ so that $\frac{4}{e^R}<\lambda_0$ and set $\lambda_k=\frac{4}{e^{kR}}$. 

Let $B_k=\{x\in X|(k+\frac{1}{2})R\leq d(x,x_0)\leq (k+\frac{3}{2})R$ be an the annulus centered at $x_0$ for each $k=1,2,3,...$. We will cover each of these $B_k$ by ``pushing in'' the cover $\mathscr{U}_k$ of the boundary. To do so, let 

\[V_{U_k}=\{\gamma(kR,(k+2)R)|\gamma \text{ is a geodesic ray with } [\gamma]\in U_k\}\]
and $\mathscr{V_k}=\cup_{U_k\in \mathscr{U}_k}V_{U_k}$. Clearly $\mathscr{V}_k$ is a cover of $B_k$. 

\begin{center}
	\includegraphics[scale=0.25]{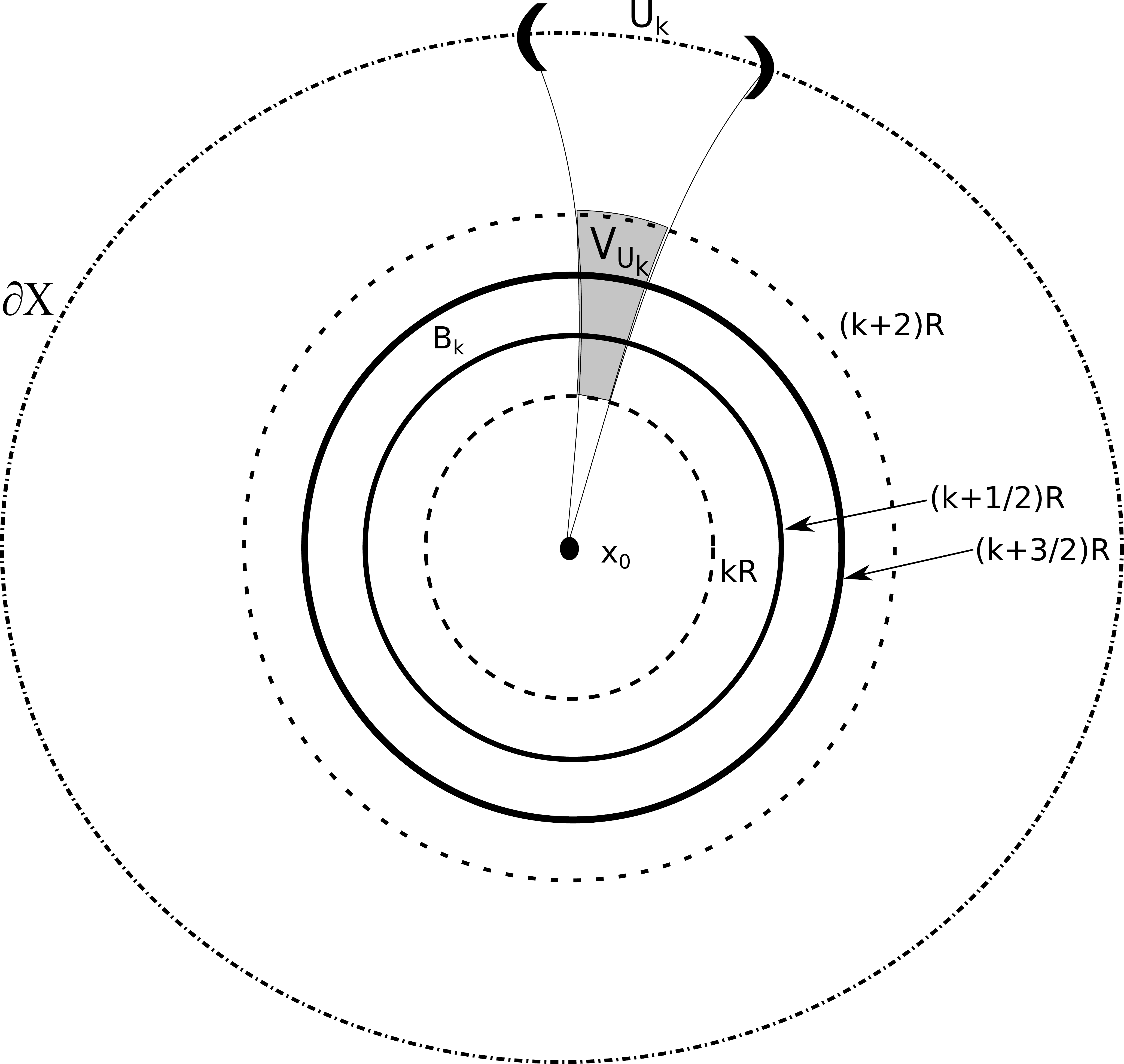}
	\end{center}

\underline{Claim 1:} $\mathscr{V}_k$ is $(n+1)$-colored by the same set $A$. That is,  $\mathscr{V}_k^a$ is a disjoint collection of sets for each $a\in A$. 

Suppose otherwise. That is, that there exists $V_U, V_{U'}\in \mathscr{V}_k^a$ with $V_U\cap V_{U'}\neq \emptyset$. If $x\in V_U\cap V_{U'}$ then there exists geodesic rays $\alpha$ and $\beta$ passing through $x$ with $[\alpha]\in U$ and $[\beta]\in U'$. Since $U,U'\in\mathscr{U}_k^a$, then $\overline{d}([\alpha],[\beta])\geq \lambda_k/2$. Thus, 
	\[\frac{\lambda_k}{2}\leq \overline{d}([\alpha],[\beta])=\int_0^\infty\frac{d(\alpha(r),\beta(r))}{e^r}\, dr\]
	\[=\int_{d(x,x_0}^\infty \frac{d(\alpha(r),\beta(r))}{e^r}\, dr\]
	\[\leq \int_{d(x,x_0}^\infty \frac{2(r-d(x,x_0)}{e^r}\, dr\]
	\[=\frac{2}{e^{d(x,x_0)}}\]
	\[<\frac{2}{e^{kR}}=\frac{\lambda_k}{2}\]
The last line provides the required contradiction. Thus, order$(\mathscr{V}_k)\leq n$ for each $k$. 

\underline{Claim 2:} For every $x,y\in V_{U_k}\in\mathscr{V}_k$ with $d(x_0,x)=(k+2)R=d(x_0,y)$, then $d(x,y)\leq 4ce^{2R}$. 
To show this, suppose otherwise. Choose $x,y\in \mathscr{V}_k$ with $d(x_0,x)=(k+2)R=d(x_0,y)$ and  $d(x,y)> 4ce^{2R}$. Let $\gamma_x$ and $\gamma_y$ be geodesic rays based at $x_0$  with $[\gamma_x],[\gamma_y]\in U_k$ and such that $\gamma_x((k+2)R)=x$ and $\gamma_y((k+2)R)=y$. Thus, 
\[\overline{d}([\gamma_x],[\gamma_y])\geq \int_{(k+2)R}^\infty\frac{d(\gamma_x(r),\gamma_y(r))}{e^r}\, dr\]
\[>\int_{(k+2)R}^\infty\frac{4ce^{2R}}{e^r}\, dr\]
\[=\frac{4c}{e^{kR}}=c\lambda_k\]

Since $[\gamma_x],[\gamma_y]\in U_k$ and mesh$\mathscr{U}_k\leq c\lambda_k$, we obtain the desired contradiction. 

\underline{Claim 3:} mesh$\mathscr{V}_k\leq 4ce^{2R}+2R$. 
Let $x,y\in V_{U_k}\in\mathscr{V}_k$. Let $\gamma_x$ and $\gamma_y$ be geodesic rays based at $x_0$ passing through $x$ and $y$, respectively. Suppose $\gamma_x(t)=x$ and $\gamma_y(s)=y$ for $t,s\in (kR,(k+2)R)$. Without loss of generality, suppose $s\leq t$. Then 
\[d(x,y)\leq d(x,\gamma_x(s))+d(\gamma_x(s),\gamma_y(s))\]
\[=(t-s)+d(\gamma_x(s),\gamma_y(s))\]
\[\leq 2R+d(\gamma_x((k+2)R),\gamma_y((k+2)R))\]
\[\leq 2R+4ce^{2R}\]

Thus, we have shown that mesh$\mathscr{V}_k\leq 4ce^{2R}+2R$ and order$\mathscr{V}_k\leq n$ for every $k$. Since $\mathscr{V}_k\cap\mathscr{V}_{k-1}=\emptyset$, then $\cup \mathscr{V}_k$ is a uniformly bounded cover of $X-B(x_0, \frac{3}{2}R)$ with order bounded above by $2n$. Letting 
$\mathscr{V}=\cup\mathscr{V}_k\cup B(x_0,2R)$ we obtain our desired cover. 

\end{proof}

The missing piece in the above argument that would prove finite asymptotic dimension is having arbitrarily large Lebesgue numbers for the cover. Thus, this argument is a potential step in finally answering the open asymptotic dimension question.


\section{Examples}

The previous sections highlight important properties and results that can be obtained using the $d_A$ and $\overline{d}$ metrics. Many of the results we obtained with the given techniques worked for one metric, but not the other. That is of course not to say that the same results cannot be obtained using different methods with the other metric. However, the different results do provide interesting comparisons between the two metrics and some insight into each ones strengths or weaknesses. In this section, we highlight some other differences by showing calculations done on $T_4$, the four valent tree.

\begin{ex} \emph{In this example, we show that $\overline{d}_{x_0}$ is a visual metric on $\partial T_4$, but $d_A$ is not a visual metric on $T_4$. }
\end{ex}
Recall that a metric $d$ on the boundary of a hyperbolic space is called a \textbf{visual metric} with parameter $a>1$ if there exists constants $k_1,k_2>0$ such that 
\[k_1a^{-(\zeta,\zeta')_p}\leq d(\zeta,\zeta')\leq k_2a^{-(\zeta,\zeta')_p}\] for all $\zeta,\zeta'\in\partial X$. 
[Here $(\zeta,\zeta')_p$ is the extended Gromov product based at $p\in X$. See \cite{BH99} for more information on visual metrics.]

Fix a base point $x_0\in X$ and $A>0$. Let $[\alpha],[\beta]\in\partial T_4$ and let $\alpha, \beta:[0,\infty)\to T_4$ be the corresponding geodesic rays based at $x_0$. Set $t=\emph{max}\{r|d(\alpha(r),\beta(r))=0\}$. Then $d(\alpha(r),\beta(r))=2(r-t)$ for all $r\geq t$. A simple computation shows:
\[\overline{d}_{x_0}([\alpha],[\beta])=\int_t^\infty\frac{2(r-t)}{e^r}\, dr=\frac{2}{e^t}\]

Furthermore, since $([\alpha],[\beta])_{x_0}=t$, we see that $\overline{d}_{x_0}$ is a visual metric on $T_4$ with parameter $e$. 

Now, suppose, by way of contradiction, that $d_A$ is visual with parameter $a>1$. Then there exists $k_1,k_2>0$ such that 
$k_1a^{-(\zeta,\zeta')_{x_0}}\leq d_A(\zeta,\zeta')\leq k_2a^{-(\zeta,\zeta')_{x_0}}$
for all $\zeta,\zeta'\in\partial X$.

Choose $n\in \ZZ^+$ large enough such that $\frac{a^n}{n+1}>k_2a^{A/2}$, which is possible since \\$\lim_{n\to\infty}\frac{a^n}{n+1}=\infty$. Let $\alpha,\beta:[0,\infty)\to X$ be any two proper geodesic rays based at $x_0$ with the property that $\alpha(t)=\beta(t)$ for all $t\leq \left \lceil{n-\frac{A}{2}}\right \rceil$ and $\alpha(t)\neq\beta(t)$ for all $t>\left \lceil{n-\frac{A}{2}}\right \rceil$ (that is, $\alpha$ and $\beta$ are two rays that branch at time $t=\left \lceil{n-\frac{A}{2}}\right \rceil$. Notice then that 
\[d_A([\alpha],[\beta])=\frac{1}{\left \lceil{n-\frac{A}{2}}\right \rceil+\frac{A}{2}} \, \text{   and   } \, ([\alpha],[\beta])_{x_0}=\left \lceil{n-\frac{A}{2}}\right \rceil\]

By the visibility assumption, 
\[d_A([\alpha],[\beta])\leq k_2a^{-([\alpha],[\beta])_{x_0}}\] 
and thus, 
\[\frac{1}{\left \lceil{n-\frac{A}{2}}\right \rceil+\frac{A}{2}}\leq k_2a^{-([\alpha],[\beta])_{x_0}}\]

Since $\left \lceil{n-\frac{A}{2}}\right \rceil\geq n-\frac{A}{2}$ and $\left \lceil{n-\frac{A}{2}}\right \rceil\leq n-\frac{A}{2}+1$, we obtain the following inequality: 

\[\frac{1}{n+1}\leq\frac{1}{\left \lceil{n-\frac{A}{2}}\right \rceil+\frac{A}{2}}\leq k_2a^{-([\alpha],[\beta])_{x_0}}=k_2a^{-\left \lceil{n-\frac{A}{2}}\right \rceil}\leq k_2a^{-(n-\frac{A}{2})}\]
Rearranging, we see that 
\[\frac{a^n}{n+1}\leq k_2a^{A/2},\]
a contradiction to the choice of $n$.

\begin{prop}$id_{\partial X}:(\partial X,d_A)\to (\partial X,\overline{d})$ is not a quasi-symmetry.
\end{prop}
 
 We prove this proposition by showing it in the case that $X=T_4$. For this, we need the following lemma. 

\begin{lemma} $(\partial T_4,d_A)$ is uniformly perfect. \end{lemma} 

\begin{proof} Fix a base point $x_0\in T_4$. It suffices to show $(\partial T_4, d_1)$ is uniformly perfect since $(\partial T_4, d_A)$ is quasi-symmetric to $(\partial T_4, d_1)$ for every $A>0$ by Lemma 3.1.3. Let $[\alpha]\in\partial T_4$ and $\alpha:[0,\infty)\to T_4$ the ray asymptotic to $[\alpha]$ based at $x_0$.  Since diam$(T_4,d_1)=2$, we show that $B([\alpha],r)-B([\alpha],\frac{r}{4})\neq\emptyset$ for all $0<r<2$.  Consider the geodesic ray $\beta:[0,\infty)\to T_4$ based at $x_0$ with $\alpha(t)=\beta(t)$ for all $t\leq \lceil \frac{1}{r}\rceil$ and $\alpha(t)\neq\beta(t)$ for all $t>\lceil\frac{1}{r}\rceil$. Then, $d_1([\alpha],[\beta])=\frac{1}{\lceil1/r\rceil+1/2}$ which means $d_1([\alpha],[\beta])<r$. Moreover, $\lceil\frac{1}{r}\rceil +\frac{1}{2}\leq \frac{1}{r}+1+\frac{1}{2}<\frac{1}{r}+\frac{3}{r}$, so $d_1([\alpha],[\beta])>\frac{r}{4}$. This proves $[\beta]\in B([\alpha],r)-B([\alpha],\frac{r}{4})$. 
\end{proof}

\begin{proof}[Proof of Proposition 5.0.3]  Let $X=T_4$. We will show that $id:(\partial T_4,d_A)\to(\partial T_4,\overline{d})$ is not a quasi-symmetry for $A=1$ and then refer to Proposition 3.1.3 for the full claim. Fix a base point $x_0\in T_4$ and suppose, by way of contradiction, that $id:(\partial T_4,d_1)\to(\partial T_4,\overline{d})$ is a quasi-symmetry. By Theorem 2.2.4 and Lemma 5.0.4,  $\eta$ must be of the form $\eta(t)=c\, \text{max}\{t^{\delta},t^{1/\delta}\}$ where $c\geq 1$ and $\delta\in(0,1]$ depends only on $f$ and $X$. Let $\alpha,\gamma:[0,\infty)\to T_4$ be two proper geodesic rays such that $\alpha(t)\neq\gamma(t)$ for all $t>0$. Then 
	\[d_1([\alpha],[\gamma])=\frac{1}{1/2}=2\]
	\[\overline{d}([\alpha],[\gamma])=\int_0^{\infty}\frac{2r}{e^r}\, dr =2\]

Choose $n\in\ZZ^+$ large enough such that $n-\frac{1}{\delta}\ln(2n+1)>\ln(c)$, which is possible since $\lim_{n\to\infty} n-\frac{1}{\delta}\ln(2n+1)=\infty$. 

Let $\beta:[0,\infty)\to T_4$ be a proper geodesic ray with the property that $\beta(t)=\gamma(t)$ for all $t\leq n$ and $\beta(t)\neq\gamma(t)$ for all $t>n$. Then 
	\[d_1([\beta],[\gamma])=\frac{1}{n+1/2}=\frac{2}{2n+1}\]
	\[\overline{d}([\beta],[\gamma])=\int_n^{\infty}\frac{2(r-n)}{e^r}\, dr =\frac{2}{e^n}\]

Set $t=\frac{d_1([\alpha],[\gamma])}{d_1([\beta],[\gamma])}=2n+1$.

By the quasi-symmetry assumption, 

\[\overline{d}([\alpha],[\gamma])\leq \eta(t)\overline{d}([\beta],[\gamma])\]
and thus, 
\[2\leq\eta(2n+1)\frac{2}{e^n}\]
\[\Rightarrow e^n\leq\eta(2n+1)=c\, \text{max}\{(2n+1)^{\delta},(2n+1)^{1/\delta}\}\]
\[\Rightarrow e^n\leq c(2n+1)^{1/\delta}\]
\[\Rightarrow n\leq \ln(c)+\frac{1}{\delta}\ln(2n+1)\]
This last inequality contradicts the choice of $n$, proving our claim. 

\end{proof}

%

\section{Open Questions}

Since metrics on visual boundaries of CAT(0) spaces have not been widely studied, there is still much work to be done in this area. We hope that the results here show the development of these metrics is worthwhile and provides the opportunity to study CAT(0) boundaries from a different point of view, which may of course lead to answering interesting unanswered questions about these boundaries. We end with a list of open questions.

\begin{question} Is there an extension of $d_A$ to $\overline{X}$ that is equivalent to the cone topology on $\overline{X}$? \end{question} 

\begin{question} In the proof of Theorem 3.1.5, a different control function is used for each $g\in G$. Is there a single control function  for the entire group? \end{question}

\begin{question} Are all of the members of the $\overline{d}_{x_0}$ family of metrics quasi-symmetric?\end{question}
The answer to this question is yes in the extreme cases that $X$ is $\RR^2$ or the four-valent tree by simple calculations. If it can be shown that the answer is yes for any CAT(0) space $X$, then we could easily show that the group of isometries of a CAT(0) space acts by quasi-symmetries on the boundary as in Theorem 3.1.5. 

\begin{question} Is the linearly controlled dimension of CAT(0) group boundaries finite when the boundary is endowed with the $\overline{d}_{x_0}$ metric? Furthermore, if the answer to Question 6.0.7 is no, can a CAT(0) boundary with two different metrics from the same family $\{\overline{d}_{x_0}\}$ have different linearly controlled dimension? \end{question} 

\begin{question} For a hyperbolic group $G$, $\ell\emph{-dim}\partial X=\emph{dim}\partial X$. Can the same be said for CAT(0) group boundaries? In particular, can it be shown for a CAT(0) group $G$, $\ell\emph{-dim}\partial X\leq \emph{dim}\partial X$ with respect to either the $d_A$ metric or $\overline{d}$ metric?
\end{question}

\begin{question} In Example 1, we showed that $\overline{d}_{x_0}$ is a visual metric on $\partial T_4$. Is $\overline{d}_{x_0}$ a visual metric on the boundary of any $\delta$-hyperbolic space? \end{question}

%

\bibliography{Biblio}{}
\bibliographystyle{amsalpha}

\end{document}